\documentclass[12pt]{amsart}
\usepackage{amsmath,amssymb,amsbsy,amsfonts,amsthm,latexsym,mathabx,
	amsopn,amstext,amsxtra,euscript,amscd,stmaryrd,mathrsfs,
	cite,array,mathtools,enumerate}

\usepackage{url}
\usepackage[colorlinks,linkcolor=blue,anchorcolor=blue,citecolor=blue,backref=page]{hyperref}

\usepackage{color}
\usepackage{float}

\hypersetup{breaklinks=true}

\usepackage[english]{babel}

\usepackage{mathtools}
\usepackage{todonotes}
\usepackage{algorithm}
\usepackage{algpseudocode}

\usepackage{enumitem}

\usepackage{mathtools}
\usepackage{todonotes}
\usepackage{url}
\usepackage[colorlinks,linkcolor=blue,anchorcolor=blue,citecolor=blue,backref=page]{hyperref}

\begin{document}
	
	\newtheorem{theorem}{Theorem}
	\newtheorem{lemma}[theorem]{Lemma}
	\newtheorem{claim}[theorem]{Claim}
	\newtheorem{cor}[theorem]{Corollary}
	\newtheorem{prop}[theorem]{Proposition}
	\newtheorem{definition}{Definition}
	\newtheorem{question}[theorem]{Open Question}
	\newtheorem{example}[theorem]{Example}
	\newtheorem{remark}[theorem]{Remark}
	
	\numberwithin{equation}{section}
	\numberwithin{theorem}{section}

	\newcommand{\F}{\mathbb{F}}
	\newcommand{\K}{\mathbb{K}}
	\newcommand{\D}[1]{D\(#1\)}
	\def\scr{\scriptstyle}
	\def\\{\cr}
	\def\({\left(}
	\def\){\right)}
	\def\[{\left[}
	\def\]{\right]}
	\def\<{\langle}
	\def\>{\rangle}
	\def\fl#1{\left\lfloor#1\right\rfloor}
	\def\rf#1{\left\lceil#1\right\rceil}
	\def\le{\leqslant}
	\def\ge{\geqslant}
	\def\eps{\varepsilon}
	\def\mand{\qquad\mbox{and}\qquad}

	\def\vec#1{\mathbf{#1}}
	
	\def\ov\Q{\overline\Q}

	\def\bl#1{\begin{color}{blue}#1\end{color}} 
	
	\newcommand{\Q}{\mathbb{Q}}
	
	\newcommand{\C}{\mathbb{C}}
	\newcommand{\Fq}{\mathbb{F}_q}
	\newcommand{\Fp}{\mathbb{F}_p}
	\newcommand{\Disc}[1]{\mathrm{Disc}\(#1\)}
	\newcommand{\Res}[1]{\mathrm{Res}\(#1\)}
	\newcommand{\ord}{\mathrm{ord}}
	\newcommand{\PrePer}{\mathrm{PrePer}}
	
	\newcommand{\Z}{\mathbb{Z}}
	\renewcommand{\L}{\mathbb{L}}
	\newcommand{\PP}{\mathbb{P}}
	\newcommand{\Char}{\mathrm{Char}}
	
	\newcommand{\Norm}{\mathrm{Norm}}
	
	\def\cA{{\mathcal A}}
	\def\cB{{\mathcal B}}
	\def\cC{{\mathcal C}}
	\def\cD{{\mathcal D}}
	\def\cE{{\mathcal E}}
	\def\cF{{\mathcal F}}
	\def\cG{{\mathcal G}}
	\def\cH{{\mathcal H}}
	\def\cI{{\mathcal I}}
	\def\cJ{{\mathcal J}}
	\def\cK{{\mathcal K}}
	\def\cL{{\mathcal L}}
	\def\cM{{\mathcal M}}
	\def\cN{{\mathcal N}}
	\def\cO{{\mathcal O}}
	\def\cP{{\mathcal P}}
	\def\cQ{{\mathcal Q}}
	\def\cR{{\mathcal R}}
	\def\cS{{\mathcal S}}
	\def\cT{{\mathcal T}}
	\def\cU{{\mathcal U}}
	\def\cV{{\mathcal V}}
	\def\cW{{\mathcal W}}
	\def\cX{{\mathcal X}}
	\def\cY{{\mathcal Y}}
	\def\cZ{{\mathcal Z}}
	
	\def\fra{{\mathfrak a}} 
	\def\frb{{\mathfrak b}}
	\def\frc{{\mathfrak c}}
	\def\frd{{\mathfrak d}}
	\def\fre{{\mathfrak e}}
	\def\frf{{\mathfrak f}}
	\def\frg{{\mathfrak g}}
	\def\frh{{\mathfrak h}}
	\def\fri{{\mathfrak i}}
	\def\frj{{\mathfrak j}}
	\def\frk{{\mathfrak k}}
	\def\frl{{\mathfrak l}}
	\def\frm{{\mathfrak m}}
	\def\frn{{\mathfrak n}}
	\def\fro{{\mathfrak o}}
	\def\frp{{\mathfrak p}}
	\def\frq{{\mathfrak q}}
	\def\frr{{\mathfrak r}}
	\def\frs{{\mathfrak s}}
	\def\frt{{\mathfrak t}}
	\def\fru{{\mathfrak u}}
	\def\frv{{\mathfrak v}}
	\def\frw{{\mathfrak w}}
	\def\frx{{\mathfrak x}}
	\def\fry{{\mathfrak y}}
	\def\frz{{\mathfrak z}}
	
	\newcommand{\fp}{\mathfrak{p}}
	\newcommand{\fq}{\frak{q}}
	\newcommand{\fa}{\frak{a}}
	\newcommand{\fb}{\frak{b}}
	\newcommand{\fc}{\frak{c}}
	\newcommand{\fd}{\frak{d}}
	\newcommand{\A}{\frak{a}}
	\newcommand{\fA}{\frak{A}}
	\newcommand{\fB}{\frak{B}}
	\newcommand{\fC}{\frak{C}}
	\newcommand{\fD}{\frak{D}}
	\newcommand{\fP}{\frak{P}}
	\newcommand{\fQ}{\frak{Q}}
	
	\newcommand{\Root}{\mathcal{R}}

	\def\ov\Q{\overline{\Q}}
	\def\ov\K{\overline{\K}}
	\def \brho{\boldsymbol{\rho}}
	
	\def \pf {\mathfrak p}
	
	\def \Prob{{\mathrm {}}}
	\def\e{\mathbf{e}}
	\def\ep{{\mathbf{\,e}}_p}
	\def\epp{{\mathbf{\,e}}_{p^2}}
	\def\em{{\mathbf{\,e}}_m}
	\def\h{\mathrm{h}}
	\def\disc{\mathrm{disc}}
	
	\newcommand{\sR}{\ensuremath{\mathscr{R}}}
	\newcommand{\sDI}{\ensuremath{\mathscr{DI}}}
	\newcommand{\DI}{\ensuremath{\mathrm{DI}}}

	\def\Orb{\mathrm{Orb}}
	\def\aOrb{\overline{\mathrm{Orb}}}

	\def \Nm{{\mathrm{Nm}}}
	\def \Gal{{\mathrm{Gal}}}
	
	\newenvironment{notation}[0]{%
		\begin{list}%
			{}%
			{\setlength{\itemindent}{0pt}
				\setlength{\labelwidth}{1\parindent}
				\setlength{\labelsep}{\parindent}
				\setlength{\leftmargin}{2\parindent}
				\setlength{\itemsep}{0pt}
			}%
		}%
		{\end{list}}
	
	\definecolor{dgreen}{rgb}{0.,0.6,0.}
	\def\tgreen#1{\begin{color}{dgreen}{\it{#1}}\end{color}}
	\def\tblue#1{\begin{color}{blue}{\it{#1}}\end{color}}
	\def\tred#1{\begin{color}{red}#1\end{color}}
	\def\tmagenta#1{\begin{color}{magenta}{\it{#1}}\end{color}}
	\def\tNavyBlue#1{\begin{color}{NavyBlue}{\it{#1}}\end{color}}
	\def\tMaroon#1{\begin{color}{Maroon}{\it{#1}}\end{color}}

	\title[Powers in orbits]{Powers in orbits and semigroups generated by polynomials} 
	\newcommand{\set}[1]{\left\{#1\right\}}
\begin{abstract}  
In this paper, we study the presence of perfect powers in orbits of polynomials, or of semigroups generated by a finite family of polynomials, defined over a number field. In particular, we correct and generalize a result of Ahmadi et al. (2012),  showing that multiplicative shifts of perfect powers in orbits can be detected by a finite algorithm. We also give a finiteness result on elements in orbits satisfying certain multiplicative relations.
\end{abstract}

\author[A. Ostafe]{Alina Ostafe}
\address{School of Mathematics and Statistics, University of New South Wales, Sydney, NSW 2052, Australia}
\email{alina.ostafe@unsw.edu.au}

\author[D. T. Tran]{Dat T. Tran}
\address{School of Mathematics and Statistics, University of New South Wales, Sydney, NSW 2052, Australia}
\email{dat.t.tran@unsw.edu.au}

\pagenumbering{arabic}

\maketitle

	\section{Motivation and statements of main results}
	\subsection{Motivation}
For a polynomial $f\in \K[X]$ over a field $\K$ we define
	the sequence of polynomials:
\begin{align*}
    	f^{(0)}(X)  = X, \qquad f^{(n)}(X)  = f(f^{(n-1)}(X)), \quad
	n =1, 2, \ldots\,.
\end{align*}
	The polynomial $f^{(n)}$ is called the $n$-th iterate of the polynomial $f$.
	
	Algebraically, the set of iterates $\{f^{(i)} : i \in \mathbb{N}\}$ forms a semigroup under composition. Extending this notion to a general non-empty family of polynomials $\mathcal{F} \subseteq \mathbb{K}[X]$, we denote by $\langle \mathcal{F} \rangle$ the semigroup generated by $\mathcal{F}$ under composition. More precisely, the elements of $\langle \mathcal{F} \rangle$ are all finite compositions of the form $f_1 \circ f_2 \circ \dots \circ f_r$, where $f_j \in \mathcal{F}$ for $1 \le j \le r$.
	
The presence of perfect powers in orbits of a polynomial $f\in \K[X]$ over a number field $\K$ has been intensively studied in the last decade. Motivated initially by studying the dynamical irreducibility of quadratic polynomials, Ahmadi, Luca, Ostafe and Shparlinski~\cite{ALOS} studied the presence of squares in the orbit of the critical point $\gamma$ of an irreducible quadratic polynomial  $f(X)=aX^2+bX+c\in \Z[X]$. They showed~\cite[Theorem 5]{ALOS} that if $f^{(k)}(\gamma)$ is a square, then
$$
k<\exp(2^{1377}H^{80}),
$$
where $H=\max(|a|,|b|,|c|,3)$. While $k$ is indeed uniformly bounded, the proof of~\cite[Theorem 5]{ALOS} uses a functional relation in the last line on page 362 of~\cite{ALOS} which is unfortunately incorrect. 

Cahn, Jones  and Spear~\cite{CJS} have also obtained a series of results about the structure 
of intersections of orbits 
of a rational function $\psi \in \K(X)$   over a  field $\L$ of characteristic zero,  
 with the image
$\varphi(\K)$ of  $\K$  for 
another rational function $\varphi \in \K(X)$.  
 In the special case  of  $\varphi(X) = X^m$, with a fixed integer $m \ge 2$, this corresponds to the case
of powers in orbits of rational functions, see~\cite[Corollaries~1.6--1.8]{CJS}.  
In particular,
Cahn, Jones  and Spear~\cite[Corollary~1.8]{CJS} show that if $f \in \K[X]$ is of degree $d_f\ge 2$ and $m\ge 2$ such that for some $\alpha \in \K$ the 
intersection of the  orbit of $\alpha$  
with the set of $m$-th powers 
$\K^m$ is  infinite, then one of the following holds:
\begin{itemize}
\item $f(X)=cX^j (g(X))^m$ for some $g \in \K[X]$, $0 \le j \le m-1$, and $c\in  \K^\times$; or
\item $m=2$ and there is a non-zero $c\in  L$ such that $cf(X/c)=(-1)^{d_f}T_{d_f}(X+2)-2$, where $T_{d_f}$ is the  monic Chebyshev polynomial of degree $d_f$.
\end{itemize}
Soon after, Ostafe, Pottmeyer and Shparlinski~\cite{OPS} consider  this question for polynomials  $f(X) \in \K[X]$  over a number field $\K$ and 
extend it to considering the union of all orbits over all   $\alpha \in \K$, and study its intersection with the set of all nontrivial powers of $S$-integers.

The goal of this note is to correct and improve~\cite[Theorem 5]{ALOS}, as well as generalise it in various ways. More precisely we extend~\cite[Theorem 5]{ALOS} to:
\begin{itemize}
\item orbits of polynomials $f\in\K[X]$, defined over a number field (rather than $\Q$) of any degree $d_f\ge 2$ (rather than quadratic) and not necessarily irreducible, at any $\alpha\in\K$ (rather than critical point) (Theorem~\ref{thm:pow});
\item orbits of semigroups generated by a finite family  of polynomials defined over a number field  $\K$ (Theorem~\ref{thm:extend_for_semigroup_final}).
\end{itemize}

Moreover, we consider multiplicative shifts of powers of $S$-integers (where $S$ is a finite set of places such that $f\in\cO_S[X]$ and $\alpha\in\cO_S$) rather than just perfect powers.
We also obtain  a finiteness result on the set of elements in orbits of semigroups that are multiplicatively dependent with a given element of $\K$ (Corollary~\ref{cor:multiplicative_dependence}). This is motivated by various works on the multiplicative dependence of elements in orbits, such as~\cite{BBGMOS,BOSS,OSSZ}.\
	
	\subsection{Notation and conventions}\label{sec:notation}
		Throughout the paper, $p$ always denotes a prime number. 
For any  positive real number $x$, we denote by $\lfloor x\rfloor$ the largest integer smaller or equal to $x$.

Let $\mathbb{K}$ be a number field and fix an algebraic closure $\overline{\mathbb{K}}$. We  denote by $\K^\times$ the multiplicative group of all non-zero elements in $\K$. Let $S$ be a finite set of places of $\mathbb{K}$ containing the archimedean ones, and $\mathcal{O}_S$ the ring of $S$-integers. For $\alpha \in \K$, 
we define the {\it $S$-norm} of $\alpha$ by
\begin{equation}  \label{eq:S-norm}
N_S(\alpha)=\prod_{v\in S}|\alpha|_v,
\end{equation}
 where  $|\cdot|_v$ is the normalised absolute valuation of $\K$ at the place $v$, see~\cite{BG06,HS00,Zan} for more details on valuations.

For a polynomial $f\in \mathbb{K}[X]$ we use the notation $d_f$ for its degree. We define $\mathcal{C}_f$ as the set of critical values of $f$, that is, \begin{align}
    \label{eq:set of critical values}\mathcal{C}_f=\left\{f(x):f'(x)=0\right\}.
\end{align} We also denote by $\Root(f)$ as the set of distinct roots of $f$ in $\overline{\K}$.
		
For $m\ge 2$, we define by $\Root_m(f)$ as the set of distinct roots of $f$ whose multiplicities are not divisible by $m$.
		
In the case that $f(X)=aX^2+bX+c$, we define
    \begin{align}\label{eq:c(f)}
			c(f)=4ac-b^2+2b.
		\end{align}
		
For each family $\mathcal{F}\subset \K[X]$, an element $\alpha\in \K$ is called a \textit{wandering point} of $\mathcal{F}$ if $W(\alpha)\ne W'(\alpha)$ for all distinct words $W,W'\in \langle \mathcal{F} \rangle$, where $\langle \mathcal{F} \rangle$ is the semigroup generated by $\mathcal{F}$ under composition.

Finally, we use the Landau symbol $O$ and the Vinogradov symbol $\ll$. Recall that the assertions $U=O(V)$ and $U \ll V$ are both equivalent to the inequality $|U|\le cV$ with some absolute constant $c>0$. The implied constants in our bounds are effectively computable and may depend on  $\K$, $S$, a fixed power integer $m$, $f\in\K[X]$ or the set  $\cF$ of polynomials in $\K[X]$ that we consider in our results.			
	
\subsection{Main results}
			
In this section and the rest of the paper we fix $\mathbb{K}$ to be a number field, unless otherwise stated, and $S$ a finite set of places of $\mathbb{K}$ containing the archimedean ones.

\begin{theorem}\label{thm:pow}
Let $f \in \mathcal{O}_S[X]$ be a polynomial of degree $d_f \ge 2$, $\alpha \in \mathcal{O}_S$ a wandering point, and $m \ge 2$ an integer. Assume that $f$ is not of the form $f(X) = X^u p(X)^m$ for some integer $u \ge 0$ and $p \in \overline{\mathbb{K}}[X]$, and if $m = 2$ then also $f^{(2)}$ is not of this form. If there exist a non-zero element $b \in \mathcal{O}_S$ and an element $y \in \mathcal{O}_S$ such that 
    \begin{equation}\label{eq:sup_equa}
	f^{(k)}(\alpha) = b y^m,
    \end{equation}
then $k$ satisfies the inequality
    \begin{equation*}
    k \le C_1\log N_S(b) + C_2,
    \end{equation*}
where $C_1$ and $C_2$ are effectively computable constants depending on $\K$, $m$,  and $f$, and $C_2$ depends on $S$ as well.
\end{theorem}
In the proof of Theorem~\ref{thm:pow} we make explicit the constants $C_1$ and $C_2$ in the bound above.

\begin{remark}
We note that, for fixed non-zero $b\in\cO_S$, if one does not fix the power $m$ in Theorem~\ref{thm:pow}, then one obtains an ineffective finiteness result. Indeed,  if $y$ in~\eqref{eq:sup_equa} is non-zero and not an $S$-unit, then by~\cite[Theorem 2.2]{BBGMOS} we obtain that 
$$
m\ll  (\log^*N_S (b)) \log^* \log^* N_S (b),
$$
where $\log^*x=\max(1,\log x)$. Therefore, in this case we can apply Theorem~\ref{thm:pow} to obtain an effective bound for $k$. If $y$ is an $S$-unit, then a non-effective bound for $k$ follows from~\cite[Theorem 1.2]{BOSS}.
\end{remark}

Since the presence of squares in orbits is intimately related to studying  the dynamical irreducibility of quadratic polynomials, see~\cite{ALOS} and references therein, we present below an explicit version of Theorem~\ref{thm:pow} for quadratic irreducible polynomials over $\Z$,  improving and correcting the bound of~\cite[Theorem 5]{ALOS}. 
Such a bound can be obtained directly from the proof of Theorem~\ref{thm:pow} (with $b=1$ and $m=2$), by computing the constants $C_1$ and $C_2$, which depend on $f$, $\K$ and the set $S$ containing $2$ and the primes dividing the leading coefficient. However, we present below a better bound which is uniform in $S$, improving as well the dependence on $H(f)$.

\begin{theorem}
\label{thm:quadratic_pol}
Let $f=aX^2+bX+c  \in \mathbb{Z}[X]$ be a quadratic polynomial that is irreducible over $\mathbb{Q}$ 
such that its critical point $\gamma$ is a wandering point. 
If there exist an integer $k \ge 2$ and an element $y \in \Q$ such that $f^{( k)}(\gamma) = y^2$, then $k$ satisfies
    \begin{align*}
        k\le 128000\log H(f) + 567047,
    \end{align*}
where $H(f) = \max(|a|, |b|, |c|)$.
	\end{theorem}
	
\begin{remark}
We note that the condition of $\gamma$ being a wandering point in Theorem~\ref{thm:quadratic_pol} cannot be removed. For example, Theorem~\ref{thm:quadratic_pol} fails for $f(X)=X^2-2X+2$, which is irreducible with critical point $\gamma=1$, that satisfies $f^{(k)}(\gamma)=1$ for all $k\ge 1$. See also Remark~\ref{rem:gamma preper}.
\end{remark}

\begin{remark}
Theorem~\ref{thm:quadratic_pol} shows that there exists a deterministic polynomial in $H(f)$ algorithm to determine all $k\ge 2$ such that $f^{(k)}(\gamma)$ is a square. 
\end{remark}

Next, we generalise Theorem~\ref{thm:pow} to semigroups. We recall the notations $\cR(f), \cR_m(f)$, and $\cC_g$ in Section~\ref{sec:notation}.

\begin{theorem}\label{thm:extend_for_semigroup_final}
 Let $\mathcal{F} \subseteq \mathcal{O}_S[X]$ be a finite family of polynomials with  $d_{\min}=\min_{f\in \mathcal{F}} \deg(f) \ge 2$ and $\alpha \in \mathcal{O}_S$ a wandering point. Fix an integer $m \ge 2$ and assume that $|\Root_m(f)|\ge 2$ and $\Root_m(f^{( k)}) \cap \mathcal{C}_g = \emptyset$ for all $1 \le k \le 5$ and distinct $f, g \in \mathcal{F}$.  If there exist a non-zero element $b \in \mathcal{O}_S$ and an element $\beta \in \cO_S$ such that $W_k(\alpha) = b\beta^m$ for some $W_k \in \langle \mathcal{F} \rangle$ of length $k\ge 1$, then 
\begin{equation*}
k \le C_3\log N_S(b) + C_4,
\end{equation*}
where $C_3$ and $C_4$ are effectively computable constants depending only on $\K$, $m$, $\cF$,   and $S$.
\end{theorem}
	
\begin{remark}
In the proof of Theorem~\ref{thm:extend_for_semigroup_final}, we will  decompose the word $W_k$ of length $k>6$ as $W_k=W_1 \circ W_2$ for some $W_1$ of length at most $6$. The conditions that $|\Root_m(f)|\ge 2$ and  $\Root_m(f^{( k)}) \cap \mathcal{C}_g = \emptyset$ for all $1 \le k \le 5$ and distinct $f, g \in \mathcal{F}$ are needed when applying Lemma~\ref{lem:suff_LeVeque} to show that $W_1$ satisfies the LeVeque condition, see Section~\ref{sec:LeVeque}.
\end{remark}
	
 For $\cF\subset \K[X]$ a finite family of polynomials and $\alpha \in \K$, we define
$$
\mathcal{O}_{\mathcal{F}}(\alpha)=\{W(\alpha) \mid W \in \langle \cF\rangle\}.
$$
We record the following consequence on the finiteness of multiplicatively dependent elements in a fixed orbit of a semigroup with a given non-zero element.
	
\begin{cor}\label{cor:multiplicative_dependence}
Let $\mathcal{F} \subseteq \mathcal{O}_S[X]$ be a finite family of polynomials with minimum degree $d_{\min} \ge 2$. Assume that $|\Root(f)|\ge 2$ for each $f \in \mathcal{F}$ and assume that for any distinct $f, g \in \mathcal{F}$
	\begin{align}
		\label{eq:additional condition} \quad \Root(f^{(k)}) \cap \mathcal{C}_g = \emptyset \quad \text{for all } 1 \le k \le 5. 
	\end{align}
Fix a wandering point $\alpha \in \mathcal{O}_S$ for the semigroup $\langle \mathcal{F} \rangle$. 
	Let $\gamma \in \mathbb{K}^\times$ and define
	\begin{align*}
		\mathcal{M}(\gamma) &= \left\{ x \in \mathbb{K}^\times \mid \exists (a,b) \in \mathbb{Z}^2\setminus{\{\vec{0}\}}, x^a = \gamma^b \right.\\
		&\qquad\qquad\qquad\qquad\qquad\left. \textrm{ and }a\ne 0 \textrm{ if }\gamma \textrm{ is a root of unity}\right\}.
	\end{align*}
	Then 	\begin{align*}
	\left|\mathcal{O}_{\mathcal{F}}(\alpha) \cap \mathcal{M}(\gamma)\right| < \infty.
	\end{align*}
	\end{cor}

\medskip
\noindent\textbf{Acknowledgements.} The authors are very grateful to Igor Shparlinski for useful comments on a previous draft of this paper. The authors were supported by the Australian Research Council Grant FT250100208. D. T. Tran was also supported by an UNSW TFS.

\section{Preliminaries}
\subsection{Heights}
Let $\K$ be a  number field of degree $d$ and $M_\K$  the set of all places in $\K$.	For $\alpha\in\K$, $\alpha \ne 0$, we denote by $\h(\alpha)$ the absolute logarithmic Weil height of $\alpha$, see~\cite{BG06,HS00} for more details on height functions.

For any polynomial $f= a_{d_f} X^{d_f} + \cdots + a_{0} \in \K[X]$ of degree $d_f \ge 1$,
as usual its \textit{absolute (multiplicative) height},
denoted by $H(f)$, is defined to be
    \begin{align*}
    	H(f)=\prod_{v\in M_\K} \max \left( 1,|a_0|_v,\dots , |a_{d_f}|_v\right)^{1/[\K:\Q]}
    \end{align*}
and its \textit{absolute logarithmic Weil height}, denoted by $\h(f)$, is defined as
	\begin{align*}
		\h(f)=\log H(f). 
	\end{align*}

We need a lower bound on the growth of the height of values of iterates, see~\cite[Theorem 3.11]{Silv} or~\cite[Proposition 3.2]{Zan}, which we present in the following form.

\begin{lemma}\label{lem:height_transformation}
	Let $f \in \overline{\mathbb{Q}}[X]$ be a polynomial of degree $d_f \ge 2$. Then, there exists an explicitly computable constant $C_f>0$ depending only on $f$, such that for all $z \in \overline{\mathbb{Q}}$,
	\begin{align*}
		h(f^{(n)}(z)) - C_f \ge d_f^n (h(z) - C_f).
	\end{align*}
\end{lemma}

\begin{remark}
The constant $C_f$ can be explicitly computed, for example by~\cite[Lemma 3.3]{Ima}, one can take
$$
C_f=\frac{d_f\log 2 +\log d_f +d_f\h(f)}{d_f-1},
$$
or compute it directly from the proof of~\cite[Proposition 3.2]{Zan}.
\end{remark}

Next, we need an upper bound for the number of elements in $\K$ of bounded height. It is clear that when $\K=\Q$ one has
$$
\#\set{x\in \Q:\h(x)\le h}\le 2e^{2h}+e^{h}.
$$
For $\K\ne \Q$, a direct application of \cite[Equation (1.4)]{Sch} 

implies the following bound:

\begin{lemma}\label{lem:boundofset of boundheight} 
For any real number $h \ge 0$, we have
$$
		\#\set{x \in \K : \h(x) \le h} \le 2^{5d+22}e^{2h d}.
$$
\end{lemma}

For each real number $h\ge 0$, we also define 
\begin{align}\label{eq:C_b}
	E_\K(h)=\left\{\begin{aligned}
		2^{5d+22}e^{2h d},\textrm{ if }\K\ne \Q\\2e^{2h}+e^h,\textrm{ if }\K= \Q 
	\end{aligned}\right..
\end{align}

\begin{lemma}
	\label{lem:needed_index}Let $f\in \mathcal{O}_S[X]$ and let $\alpha\in \K$ be a wandering point of $f$. Then, there exists $k\le \lfloor E_\K\left(C_f+1\right)\rfloor$ such that \begin{align*}
		\h(f^{(k)}(\alpha))> C_f+ 1, 
	\end{align*}
	where $C_f$ is defined as in Lemma~\ref{lem:height_transformation}. Moreover, for $j\ge k$, we have \begin{align*}
		\h(f^{(j+1)}(\alpha))>\h(f^{(j)}(\alpha))>C_f+1.
	\end{align*}
	\end{lemma}
	
\begin{proof}
	By Lemma \ref{lem:boundofset of boundheight}, the cardinality of $\mathcal{T}=\set{x\in \K:h(x)\le C_f+1}$ is bounded by $N=\lfloor E_\K(C_f+1)\rfloor$.
	
	We now consider $\alpha,f(\alpha), f^{(2)}(\alpha),\cdots, f^{(N)}(\alpha)$. If all these elements belong to $\mathcal{T}$, the  Pigeonhole Principle guarantees that there exists $f^{(i)}(\alpha)=f^{(j)}(\alpha)$ for some $0\le i<j\le N$, which directly leads to $\alpha$  not being a wandering point. Therefore, there exist $k\le N$ such that $f^{(k)}(\alpha)\notin \mathcal{T}$ and thus, $h\left(f^{(k)}(\alpha)\right)> C_f+1.$
	
	The last statement follows by induction, and a direct application of Lemma \ref{lem:height_transformation}. Indeed, we assume the inequality true for $j\ge k$, and we prove it for $j+1$. By Lemma \ref{lem:height_transformation}, we have
	\begin{align*}
	\h(f^{(j+1)}(\alpha))\ge d_f\h(f^{(j)}(\alpha))-(d_f-1)C_f,
\end{align*} 
and hence, 
\begin{align*}
	\h(f^{(j+1)}(\alpha))-\h(f^{(j)}(\alpha))\ge (d_f-1)(\h(f^{(j)}(\alpha)) - C_f).
\end{align*}
Applying now the induction hypothesis,  we conclude the proof.
\end{proof}

\subsection{Bounds for solutions of superelliptic equations}
\label{sec:LeVeque}

Let $\K$ be a number field with discriminant $D_\K$ and set of places $M_\K$. Let $S$ be a finite subset of $M_\K$ containing all infinite places of cardinality $s=|S|$.

We  define
\begin{eqnarray}\label{eq:P,Q defined}
	&&P_S=Q_S = 1, \ \ \mbox{if $S$ consists only of infinite places;}
\nonumber	\\
	&&P_S = \max_{i=1, \ldots, s^\prime} N_K(\fp_i),\ \ Q_S = \prod_{i=1}^{s^\prime} N_K(\fp_i),
	\\
	&&\qquad\qquad\mbox{if $\fp_1,\dots,\fp_{s^\prime}$ are the prime ideals in $S$.}\nonumber
\end{eqnarray}

Let
\begin{align}\label{eq:pol}
	f(X) = a_{d_f}X^{d_f}+a_{d_f-1}X^{d_f-1}+\cdots+a_{0} \in \mathcal{O}_S[X]
\end{align}
be a polynomial of degree $d_f \geq 2$. Let $b$ be a non-zero element of $\mathcal{O}_S, m\geq 2$ an integer and consider the equation
\begin{equation}\label{eq:super}
	f(x)=by^m \qquad  \text{ in } x,y \in \mathcal{O}_S.
\end{equation}

We assume that in some finite extension of $\K$, the polynomial \eqref{eq:pol} can be factorized as follows
\begin{align*}
	f(X) = a_{d_f}\cdot \prod_{i=1}^r (X-\alpha_i)^{e_i}
\end{align*}
with distinct $\alpha_1, \dots, \alpha_r$. Put
\begin{align*}
	m_i=\frac{m}{\gcd(m,e_i)}, \ \ i=1,\dots , r.
\end{align*}
We may and shall assume below that
\begin{equation}\label{eq:m_decreasing}
	m_1\geq m_2 \geq \cdots \geq m_r.
\end{equation}
Let
\begin{align*}
	f^*(X) = a_{d_f}\cdot \prod_{i=1}^r (X-\alpha_i).
\end{align*}

It is clear that the polynomial $f^*$ has its coefficients in $O_S$.

We say that $(f,m)$, satisfies the {\it LeVeque condition}, or $f$ satisfies LeVeque with respect to $m$, if  $(m_1,\dots, m_r)$ 
is of the form 
\begin{align*}
m_1\geq 3, m_2\geq 2 \quad\quad \text{or} \quad\quad  m_1=m_2=m_3=2.
\end{align*}

For the proof of Theorems~\ref{thm:pow} and~\ref{thm:quadratic_pol} we need the following result which is a direct consequence of~\cite[Theorem 2.1 and Lemma 3.3]{BBGMOS}, see also~\cite[Theorem 2.1]{BEG} for a different bound under more restrictive conditions on $f$. We note that although  the bounds of~\cite[Theorem 2.1]{BBGMOS} are explicit, we present this result in this explicit form only for $m=2$, while for $m\ge 3$ we only give the explicit dependence on $N_S(b)$, as it is needed for our results.
	
\begin{lemma}\label{lem:T_LeVeque}
Assume that $(f,m)$ satisfies the  LeVeque condition. 
    Then all solutions $(x,y)$ of \eqref{eq:super} satisfy

    \begin{enumerate}
    \item[(i)] If $m_1=m_2=m_3=2$, then all solutions $(x,y)$ of \eqref{eq:super} satisfy
        \begin{equation}\label{eq:bound222}
            \h(x) \leq (16r^3 s)^{80r^4s}|D_\K|^{8r^3}H(f^*)^{50dr^4} Q_S^{20r^3}N_S(b)^{16r^3}.
            \end{equation}
    
    \item[(ii)] Let $m\geq 3$ and assume the condition of (i) is not satisfied. Then
        \begin{equation}\label{eq:bound32}
            \h(x) \le C N_S(b)^{4m^8r^2},
            \end{equation}
            where $C$ is a constant depending only on the $\K$, $f$, $m$, and $S$,  explicitly given in~\cite[Equations (2.9) and (2.10)]{BBGMOS}.
    \end{enumerate}
\end{lemma}

 To apply Lemma~\ref{lem:T_LeVeque} to iterations of polynomials, we also need the following result of Young~\cite[Lemma 2.3]{You}, which shows that, unless the polynomial $f$ is special, for some integer $m\ge 2$, some small iterate of $f$ together with $m$ satisfies the LeVeque condition.
  
 \begin{lemma} \label{lem:Young}
 	Let $f\in  \mathcal{O}_S[X]$ be of degree at least two and $m\ge 2$ an integer.  If f is not of the form $f(X) = X^sp(X)^m$ for some
 	integer $s \ge 0$ and $p\in \overline\K[X]$, then either $m = 2$ and $f^{(2)}$ is of this form, or $(f^{(j)},m)$
 	satisfies the LeVeque condition for some $j \le 6$.
 \end{lemma}

 Clearly, with each integer $m\ge 2$, a polynomial $f$ has the property that $(f,m)$ satisfies the LeVeque condition if $|\Root_m(f)|\ge 3$, where we recall the notation $\Root_m(f)$ in Section~\ref{sec:notation}.

\subsection{Structure of roots of polynomial iterates} 
In this sub-section, we collect a few simple results about the structure of roots of compositions of polynomials that will be needed for the proof of Theorem~\ref{thm:quadratic_pol}

 \begin{lemma}
	\label{lem:second_iterate_parity}
	Let $f \in \ov\K[X]$ be a polynomial of degree $d_f \ge 2$. If $f$ possesses at least two distinct roots with odd multiplicities, then its second iterate $f^{(2)} = f\circ f$ also possesses at least two distinct roots with odd multiplicities. 
\end{lemma}
\begin{proof}
	Assume that $f$ has at least two distinct roots with odd multiplicities, denoted as $\alpha$ and $\beta$. We can factor $f$ as
	\begin{align*}
		f(X) = (X-\alpha)^r (X-\beta)^s g(X),
	\end{align*}
	where $r, s \ge 1$ are odd integers, and $g$ is some polynomial in $\ov\K[X]$ such that $g(\alpha)\cdot g(\beta)\ne 0$.
	
	Composing $f$ with itself, we obtain the second iterate
	\begin{align*}
		 f^{(2)}(X) = (f(X)-\alpha)^r (f(X)-\beta)^s g(f(X)).
	\end{align*}
Because the exponents $r$ and $s$ are odd and $f-\alpha,f-\beta$ and $g\circ f$ do not share  common roots, the total number of odd-multiplicity roots of $f^{(2)}$ is at least the sum of the number of odd-multiplicity roots of $f-\alpha$ and $f-\beta$. Thus, it suffices to prove $(f-\alpha)(f-\beta)$ has at least $2$ odd-multiplicity roots.

Suppose  that $(f-\alpha)(f-\beta)$ has exactly one root of odd multiplicity. This implies that one of the polynomials, say $f-\alpha$, has exactly one odd-multiplicity root, while the other polynomial, $f-\beta$, has zero odd-multiplicity roots.

Note that the degree of $f-\beta$ is even as all roots have an even multiplicity. Similarly, $f-\alpha$ has exactly one root with an odd multiplicity, and all remaining roots have even multiplicities and thus, $\deg(f-\alpha)$ is odd. This is impossible as $\deg(f-\alpha)=\deg(f-\beta)$. 
    
Suppose now that $(f-\alpha)(f-\beta)$ has zero roots of odd multiplicity. This forces both $f-\alpha$ and $f-\beta$ to have only even-multiplicity roots. Thus, we can write them as follows
    \begin{align*}
        f(X) - \alpha = q_1(X)^2 \quad \text{and} \quad f(X) - \beta = q_2(X)^2, 
    \end{align*}
for some polynomials $q_1, q_2 \in \ov\K[X]$. It follows that
    \begin{align*}
        f'(X) = 2q_1(X)q_1'(X) \quad \text{and} \quad f'(X) = 2q_2(X)q_2'(X).
    \end{align*}
This implies that $q_1 \mid f'$ and $q_2\mid f'$. Since $\alpha \neq \beta$, the polynomials $q_1$ and $q_2$ are coprime. Consequently, we obtain $q_1 q_2\mid f'.$

However, we also have $\deg(q_1) = d_f/2$ and $\deg(q_2) = d_f/2$, meaning $\deg(q_1 q_2) = d_f$,  
Since $f$ is a polynomial of degree $d_f \ge 2$, its non-zero derivative $f'$ has degree exactly $d_f - 1$. Therefore, we obtain a contradiction, which completes the proof.

\end{proof}

For the next result, we recall the notation of $\Root_m(f)$ in Section~\ref{sec:notation}.

\begin{prop}\label{prop:preservation_valid_roots}
Let $f \in \ov\K[X]$ be a polynomial of degree $d_f \ge 2$. Let $c_1,c_2\in \ov\K$ be distinct and $m_1,m_2\ge 2$ integers. Then, we have
    \begin{align*}
	\left|\cR_{m_1}(f-c_1) \cup \cR_{m_2}(f-c_2)\right| \ge 2.	
    \end{align*}
\end{prop}

\begin{proof} 
The conclusion $\left|\cR_{m_1}(f-c_1) \cup \cR_{m_2}(f-c_2)\right| \ge 2$ is equivalent with showing that the number of roots of $f-c_1$ with multiplicity not divisible by $m_1$ and of roots of $f-c_2$ with multiplicity not divisible by $m_2$ is at least $2$. Suppose, for a contradiction, that the total number of such roots is at most $1$. Since $c_1\neq c_2$, $f-c_1$ and $f-c_2$ have disjoint sets of roots. Hence one of the two polynomials $f-c_i$ has no root whose multiplicity is not divisible by $m_i$, and the other has at most one such root. Relabelling $c_1$ and $c_2$, if necessary, we may assume that $f-c_1$ has at most one  root with multiplicity not divisible by $m_1$, while 
$f-c_2$ has all roots with multiplicity divisible by $m_2$.
	
Since $\ov\K$ is algebraically closed, any non-zero scalar factor can be absorbed into a $m_i$-th power. Thus the first condition implies that there exist $q_1\in \ov\K[X]$, an integer $s$ with $0\le s<m_1$, and, if $s>0$, a point $r\in \ov\K$, such that
$f(X)-c_1=(X-r)^sq_1(X)^{m_1},$ for some $q_1\in \overline{\K}[X]$ of degree $\dfrac{d_f-s}{m_1}$.

Similarly, 	$f(X)-c_2=q_2(X)^{m_2}$ for some $q_2\in \ov\K[X]$ of degree $\frac{d_f}{m_2}.$

Set
$A=f-c_1,\qquad B=c_1-c_2,\qquad C=f-c_2.$
Then
$A+B=C$, $B\neq 0$ is constant, and $A$ and $C$ have no common root since $c_1\ne c_2$. Hence $A,B,C$ are pairwise coprime. By the Mason-Stothers theorem (proved first by \cite{Stothers1981}, and then independently by Mason~\cite{Mas} and Silverman~\cite{Si84}), one has
\begin{equation}
\label{eq:abc deg}
    d_f=\max\{\deg A,\deg B,\deg C\}\le N_0(ABC)-1,
\end{equation}
 where $N_0(P)$ denotes the number of distinct roots of $P$. Since $B$ is a non-zero constant and $A,C$ are coprime, we must have $N_0(ABC)=N_0(A)+N_0(C).$
From the above factorizations, we obtain $N_0(A)\le \deg q_1+1=\frac{d_f-s}{m_1}+1,$ and
$N_0(C)\le \deg q_2=\frac{d_f}{m_2}.$	Therefore, from~\eqref{eq:abc deg} we obtain
\begin{align}\label{eq:in_root}
    d_f(m_1m_2-m_1-m_2)+m_2s\le 0.
\end{align} 

Since $d_f, m_1, m_2\ge 2$, and $s\ge 0$, the inequality \eqref{eq:in_root} occurs only when
$m_1=m_2=2$ and $ s=0.$

In this remaining case, both $f-c_1$ and $f-c_2$ are squares. Hence $N_0(A)\le \frac{d_f}{2}$ and $ N_0(C)\le \frac{d_f}{2}.$
Thus
$N_0(ABC)\le d_f.$
Then,~\eqref{eq:abc deg}  gives
$d_f\le N_0(ABC)-1\le d_f-1,$
a contradiction, which concludes the proof.
\end{proof}

\begin{cor}\label{cor:preservation_iteration}
		Let $f \in \ov\K[X]$ be a polynomial of degree $d_f \ge 2$, and let $m \ge 2$ be an integer. If $|\Root_m(f)| \ge 2$, then for any integer $j \ge 1$, we have $|\Root_m(f^{(j)})| \ge 2$.
	\end{cor}
	
\begin{proof}
    We proceed by induction on $j$. The base case $j=1$ holds trivially by the hypothesis. 
    Assume the statement holds for some $j \ge 1$, meaning that $|\Root_m(f^{(j)})| \ge 2$. Thus, there exist at least two such distinct roots, say $c_1, c_2 \in \Root_m(f^{ (j)})$, whose respective multiplicities $e_1$ and $e_2$ as roots of $f^{(j)}$ are not divisible by $m$. Let
    $$
    m_i=\frac{m}{\gcd(m,e_i)},\quad i=1,2.
		$$
Since $m\not\mid e_i$, we have $m_i\ge 2$, $i=1,2$.
		
Consider the next iterate $f^{ (j+1)} = f^{( j)}(f)$. The roots of $f^{ (j+1)}$ include the roots of the polynomials $f - c_1$ and $f - c_2$. We need to show that among the roots of $f - c_1$ and $f - c_2$ there are at least two roots, 
say $x_1$ and $x_2$ with respective multiplicities $d_1$ and $d_2$, such that $m\not\mid e_id_j$ for some $i,j=1,2$, which is equivalent with $m_i\not\mid d_j$. 		
This follows from a direct application of Proposition \ref{prop:preservation_valid_roots}, which implies that for any $i,j=1,2$ one has \begin{align*}
	\left|\cR_{m_i}(f-c_1) \cup \cR_{m_j}(f-c_2)\right| \ge 2.
\end{align*} We therefore conclude the proof.
	\end{proof} 
	
 We recall the notation of $\mathcal{C}_g $ as in \eqref{eq:set of critical values}. 
	\begin{lemma}\label{lem:suff_LeVeque}Let $m \ge 2$ be an integer. Let $f, g \in \ov\K[X]$ be polynomials with $d_g \ge 2$. Assume that $\Root_m(f) \cap \mathcal{C}_g = \emptyset$. Then we have
 	\begin{align*}
 		|\Root_m(f \circ g)| \ge \deg(g)|\cdot\Root_m(f)|.
 	\end{align*}
 \end{lemma}
 \begin{proof} 	If $\Root_m(f)$ is empty, the inequality holds trivially. Assume $\Root_m(f)$ is non-empty, and let $\beta \in \Root_m(f)$ with multiplicity $e \not\equiv 0 \pmod m$. We can write $f(X) = (X - \beta)^e q(X)$ for some polynomial $q \in \ov\K[X]$ such that $q(\beta) \neq 0$. 
 	
 We also have 	\begin{align}\label{eq:factor fg}
 		(f \circ g)(X) = (g(X) - \beta)^e q(g(X)). 
 	\end{align}
 	
 	Since $\beta \in \Root_m(f)$ and $\Root_m(f) \cap \mathcal{C}_g = \emptyset$, it follows that $\beta$ is not a critical value of $g$, which implies that the polynomial $g - \beta$ has no multiple roots, and thus any root of $g-\beta$ belongs to $\Root_m(f \circ g)$.
    
 	Furthermore, for any two distinct elements $\beta_1, \beta_2 \in \Root_m(f)$, the polynomials $g -\beta_1$ and $g-\beta_2$ cannot share a root. Thus, one has \begin{align*}
 	|\Root_m(f \circ g)|\ge\sum_{\beta\in \Root_m(f)}d_g\ge |\Root_m(f)|\cdot d_g.	
 	\end{align*}
 	This completes the proof.
 \end{proof}

 \subsection{Some preliminaries for quadratic polynomials}
 We start with a simple necessary and sufficient condition for the critical point of a quadratic polynomial to be wandering.
 
 \begin{lemma}\label{lem:wandering_critical_point}
  	Let $f=aX^2+bX+c\in\Z[X]$ be a quadratic polynomial and let $\gamma$ be its critical point. Then $\gamma$ is a wandering point of $f$ if and only if $c(f) \notin \{0, -4, -8\}$, where $c(f)$ is defined as in \eqref{eq:c(f)}.
 \end{lemma}
 \begin{proof} 
Let $ \phi(X) = aX + \frac{b}{2}$. Then we define $$g(X)=\phi\circ f\circ \phi^{-1}=X^2+C(f),$$ where $C(f)=\dfrac{c(f)}{4}$.

 Since $\phi^{-1}\circ g^{(k)}\circ \phi=f^{(k)}$ for any $k\ge 1$, we note that $\gamma$ is a wandering point for $f$ if and only if $\phi(\gamma)=0$ is a wandering point for $g(X)$.

We note that $c(f)\in \{0,-4,-8\}$ corresponds to $g(X)\in\{X^2,X^2-1,X^2-2\}$, for which $0$ is pre-periodic. Therefore, we now assume that $c(f) \notin \{0, -4, -8\}$ and we want to show that $0$ is a wandering point for $g$.

Since $c(f) \in \mathbb{Z}$, we divide the analysis into two cases based on the parity of $c(f)$.

If $c(f)$ is not divisible by $4$, then $C(f) \notin \mathbb{Z}$, and thus we can write $C(f) = \frac{A}{B}$ in its lowest terms, where the numerator $A$ is an odd integer and the denominator $B$ satisfies $B \in \{2, 4\}$.
By induction, one can prove that the denominator of $g^{(n)}(0)$, expressed in lowest terms, is equal to $B^{2^{n-1}}$ for each positive integer $n$ and whence $\gamma$ must be a wandering point of $f$.

If $c(f)$ is divisible by $4$, then $C(f) \in \mathbb{Z}$. The hypothesis $c(f) \notin \{0, -4, -8\}$ corresponds to $C(f) \notin \{0, -1, -2\}$. We consider the two sub-cases for the integer $C(f)$.

\textit{Case 1: $C(f) \ge 1$.} We have $g(0)= C(f) \ge 1$. For any $n \ge 1$, we consider the difference \begin{align}\label{eq:consec_diff}
		g^{(n+1)}(0) - g^{(n)}(0) = (g^{(n)}(0))^2 - g^{(n)}(0) + C(f).
	\end{align}
	Since $C(f) \ge 1$ and  $t^2 - t + 1 > 0$ for all real numbers $t$, we have $g^{(n+1)}(0) - g^{(n)}(0) > 0$. The orbit is therefore a strictly increasing sequence, ensuring all points are distinct.

\textit{Case 2: $C(f)\le -3$.} We first have
	\begin{align*}
		g^{(2)}(0) = C^2(f) + C(f) = |C(f)|^2 - |C(f)| = |C(f)|(|C(f)| - 1). 
	\end{align*}
	Since $C(f) \le -3$, we have $|C(f)| \ge 3$, which implies $(|C(f)| - 1) \ge 2$. Thus, one has $g^{(2)}(0) \ge 2|C(f)| > |C(f)|$, and in particular $g^{(2)}(0)>g(0)$. By induction on $n$ same computation shows that $ g^{(n)}(0)> |C(f)|$ for all $n\ge 2$.
	
	Now we prove by induction that for all $n \ge 2$, the sequence $g^{(n)}(0)$  is strictly increasing. Consider the difference as in \eqref{eq:consec_diff}.	Let $q(t) = t^2 - t + C(f)$. For $t \ge |C(f)| \ge 3$, the quadratic function $q(t)$ is increasing (its derivative $2t-1 > 0$). Since we established $g^{(n)}(0) \ge |C(f)| = -C(f)$, we have
	\begin{align*}
		g^{(n+1)}(0) - g^{(n)}(0) = q\(g^{(n)}(0)\) \ge q(-C(f)) =  C(f)(C(f)+2)>0.
	\end{align*}
	Therefore, $ g^{(n+1)}(0) > g^{(n)}(0)$ for all $n\ge 1$, which concludes the proof.
 \end{proof}
 
 The following simple  fact is well-known, but for completeness we provide a proof.
 
\begin{lemma}\label{lem:irreducible_leveque}
	Let $f \in \mathbb{K}[X]$ be a quadratic polynomial. If $f$ is irreducible over $\mathbb{K}$, then $f^{(2)}$ is square-free. In particular, $(f^{(2)},2)$ satisfies the LeVeque condition.
\end{lemma}

\begin{proof}
Let $g(X) = f^{(2)}(X)$. Since $f$ is of degree $2$, $g$ is a polynomial of degree $4$.

	Assume, for the sake of contradiction, that $g$ possesses a multiple root $\beta \in \ov\K$. This occurs if and only if $\beta$ is a common root of both $g$ and its formal derivative $g'$. Thus, we have the system
	\begin{align}
		g(\beta) &= f(f(\beta)) = 0, \label{eq:Px0}\\
		g'(\beta) &= f'(f(\beta)) \cdot f'(\beta) = 0. \label{eq:P_prime_x0}
	\end{align}
	Let $\gamma \in \mathbb{K}$ be the unique critical point of $f$ (i.e., $f'(\gamma) = 0$). Since $f$ has coefficients in $\mathbb{K}$, it follows trivially that $\gamma \in \mathbb{K}$. From equation \eqref{eq:P_prime_x0}, we encounter two distinct possibilities as follows.
	
	\textit{Case 1: $f'(\beta) = 0$.} 
	This immediately implies $\beta = \gamma$. Substituting this into equation \eqref{eq:Px0} yields that $f(\gamma)\in\K$ is a root of $f$, which contradicts the irreducibility of f.
	
	\textit{Case 2: $f'(f(\beta)) = 0$.}
	This implies $f(\beta) = \gamma$. Substituting this back into equation \eqref{eq:Px0} yields 
	that $\gamma$ is a root of $f$, which again contradicts the irreducibility of f.

Therefore, $g= f^{(2)}$ has no multiple roots and thus $(f^{(2)},2)$ satisfies the LeVeque condition.
\end{proof}

 \section{Proofs of main results}
 \subsection{Proof of Theorem~\ref{thm:pow}}
 We seek an upper bound for $k$. If $k \le 6+\lfloor E_\K(C_f+1)\rfloor$, the inequality trivially holds by suitably enlarging the constant $C_2$. Thus, we may assume $k > 6+\lfloor E_\K(C_f+1)\rfloor$, where $C_f$ is as in Lemma~\ref{lem:height_transformation} and $E_\K(h)$ as in~\eqref{eq:C_b}.
 
 By the hypothesis on $f$ and Lemma \ref{lem:Young}, there exists an integer $j \le 6$ such that the pair $(f^{(j)}, m)$ satisfies the LeVeque condition.

 We have \begin{equation*}
 	f^{(k)}(\alpha)=f^{(j)}(f^{(k-j)}(\alpha)) = b y^m,
 \end{equation*}
 which we can rewrite as
 \begin{equation*}
 	F(x) = b y^m,
 \end{equation*}
 where $x = f^{(k-j)}(\alpha)$ and $F(X)= f^{(j)}(X)$. 
 
 Since $(F,m)$ satisfies the LeVeque condition, we can directly apply Lemma~\ref{lem:T_LeVeque} to this equation and whence one obtains
 \begin{equation}\label{eq:upper_bound}
 	\h(x) \le C\cdot N_S(b)^{e(m,r)},
 \end{equation}
 where $e(m,r)=\max(16r^3,4m^8r^2)$ and the constant $C$ is defined by~\eqref{eq:bound222} and~\eqref{eq:bound32}.

We define $N=\lfloor E_\K(C_f+1)\rfloor$.  Since $k>6+N$ implying that $k-j> N$, by Lemmas~\ref{lem:height_transformation} and~\ref{lem:needed_index}, we obtain
\begin{align*}
 	\h\left(f^{(k-j)}(\alpha)\right)-C_f\ge d^{k-j-N}_f\left(\h\(f^{(N)}(\alpha)\)-C_f\right)> d_f^{k-j-N}.
 \end{align*} 
Since $j\le 6$, one obtains
\begin{align}
\label{eq:lower_bound}\log \h(x)\ge (k-6-N)\log d_f.
 \end{align}
 
On the other hand,  by~\eqref{eq:upper_bound} we have \begin{align}\label{eq:upbound}
 	\log \h(x)\le e(m,r)\log N_s(b)+\log C.
 \end{align}

 Combining~\eqref{eq:lower_bound} and~\eqref{eq:upbound}, we finally obtain \begin{align*}
 	k\le N+6+\dfrac{e(m,r)\log N_s(b)+\log C}{\log d_f},
 \end{align*} 
which concludes the proof.

\subsection{Proof of Theorem~\ref{thm:quadratic_pol}}We use same transformation as in the proof of Lemma~\ref{lem:wandering_critical_point}, that is, let $ \phi(X) = aX + \frac{b}{2}$. Then, with the integer  $c(f)$ defined as in \eqref{eq:c(f)}, let $g(X)=\phi\circ f\circ \phi^{-1}=X^2+C(f)$, where $C(f)=c(f)/4$. Therefore, we have  $\phi^{-1}\circ g^{(k)}\circ \phi=f^{(k)}$ for any $k\ge 1$.
				
				Since $\gamma$ is a wandering point, by Lemma \eqref{lem:wandering_critical_point}, one has $C(f) \notin \{0, -1, -2\}$. 

Let $f^{(k)}(\gamma)=y^2$ for some $y\in \Q$ and $k\ge 1$. We may assume that $k\ge 7$.  From the above we obtain $g^{(k)}(0)=\phi\(f^{(k)}(\gamma)\)=ay^2+\frac{b}{2}$, which we write 				
\begin{align*}
        a\(g^{(k)}(0)-\frac{b}{2}\)=(ay)^2.
\end{align*}
Writing $g^{(k)}(0)=g^{(2)}(g^{(k-2)}(0))$ and denoting $S=\set{\infty,2}$, we obtain that $F(g^{(k-2)}(0))=(ay)^2$, where $ay\in \cO_S$ and
\begin{align*}    	F(X)=a\(g^{(2)}(X)-\frac{b}{2}\)\in\cO_S[X].
\end{align*}
Simple height computation shows that 
\begin{equation}
\label{eq:height F}
H(F)\le 85H(f)^5,
\end{equation}
where we recall $H(f)=\max\{|a|,|b|,|c|\}$. Indeed,
\begin{align*}
H(F)&=\prod_{v\in M_\K}\max(1,|a|_v,|2aC(f)|_v,|a(C(f)^2)+C(f)-b/2|_v)^{1/[\K:\Q]}\\
&=\prod_{v\in M_\K}\max(|16|_v,|16a|_v,|8ac(f)|_v,|a(c(f)^2)+4c(f)-8b|_v)^{1/[\K:\Q]}.
\end{align*}

Since for any $v\in M_\K$, one has $|c(f)|_v\le 7H(f)^2$, simple computation shows that $|a(c(f)^2)+4c(f)-8b|_v\le 85H(f)^5$ dominates. 
				
We note that $F\in\cO_S[X]$ and $g^{(k-2)}(0) \in\cO_S$ (in our case $\cO_S=\Z[\frac{1}{2}]$). It is easy to check that $F$ has four distinct roots. Indeed, $F$ has a double root only when $g(0)=b/2$ or $g^{(2)}(0)=b/2$. The first condition is equivalent with the discriminant of $f$ being zero, while the second with the discriminant being a square. In both cases $f$ would be reducible, which contradicts our assumption.
Therefore, $(F,2)$  satisfies LeVeque condition with $m_1=m_2=m_3=m_4=2$. Using Lemma~\ref{lem:T_LeVeque} and taking into account that $F^*=F$, $r=4,s=2,d=1,D_K=1,Q_S=2,b=1$, we obtain 
\begin{align}\label{eq:upbound y}
    \h\(g^{(k-2)}(0) \)\le 2^{451840}H(F)^{12800}.
\end{align}

Let $c_0=\log|c(f)|+3\log 2$. To obtain a lower bound for $\h(g^{(k-2)}(0) )$, we prove first that 
\begin{equation}
    \label{eq:g^5 0}
    h(g^{(5)}(0))\ge c_0+1.
\end{equation}

Indeed, suppose first that $|c(f)| \ge 12$, which is equivalent with $|C(f)|\ge 3$. Then $|g^{(2)}(0)| = |C(f)||C(f)+1| \ge |C(f)|(|C(f)|-1)\ge 2|C(f)|$. By induction, similar computation shows that $|g^{(j)}(0)|\ge j|C(f)|^{j-1}$ for any $j\ge 2$. Therefore, since $|c(f)|\ge 12$, we obtain
$$
\log |g^{(5)}(0)|\ge 4\log|C(f)|+\log 5 \ge c_0+1.
$$
Since $h(g^{(5)}(0))\ge \log |(g^{(5)}(0))|$, the inequality~\ref{eq:g^5 0} follows in this case.

For $|c(f)| <12$, since $c(f) \notin \{0, -4, -8\}$ and $c(f)\in \Z$,   a direct computation leads to the conclusion that $h(g^{(5)}(0)) \ge c_0+1$.

Next we prove that for any $j\ge 5$, one has
\begin{equation}
    \label{eq:g^j0}
    \h(g^{(j)}(0))-c_0\ge 2^{j-5}\left(\h(g^{(5)}(0))-c_0\right).
\end{equation}

Indeed, one has $$\h\(g^{(j+1)}(0)-C(f)\)=\h\(g^{(j)}(0)^2\)=2\h\(g^{(j)}(0)\)$$ and 
$$\h\(g^{(j+1)}(0)-C(f)\)\le \h\(g^{(j+1)}(0)\)+\h( C(f)) +\log 2\le \h\(g^{(j+1)}(0)\) +c_0.$$ 
This implies 
$$\h\(g^{(j+1)}(0)\)-c_0\ge 2\left(\h\(g^{(j)}(0)\)-c_0\right),$$ which proves~\eqref{eq:g^j0}.

Combining~\eqref{eq:upbound y} and~\eqref{eq:g^j0}, we obtain 
\begin{align*}
    2^{451840}H(F)^{12800}\ge 2^{k-7}(\h\(g^{(5)}(0)\)-c_0)\ge 2^{k-7}
\end{align*}or \begin{align*}
    k\le 25600\log H(F)+451847.
\end{align*}
Inserting the bound~\eqref{eq:height F} above, we complete the proof.

\begin{remark}
\label{rem:gamma preper}
	The condition that $\gamma$ is wandering, or equivalently that $C(f) \notin \{0, -1, -2\}$ in the proof above, is necessary. Note that, since $F\(g^{(k-2)}(0)\) = a^2f^{(k)}(\gamma)$,  $f^{(k)}(\gamma)$ is a square if and only if $F\(g^{(k-2)}(0)\)$ is a square. If $C(f) \in \{0, -1, -2\}$ the result is false in general:
	
	\begin{itemize}
		\item If $C(f) = 0$, then $g^{(k)}(0)= 0$ for all $k \ge 1$, and therefore  $F\(g^{(k-2)}(0)\)=-ab/2$ for all $k \ge 2$.
		\item If $C(f) = -1$, then $g^{(2k)}(0) = 0$, for all $k \ge 1$, and therefore  $F\(g^{(2k-2)}(0)\)=-ab/2 $ for all $k \ge 2$.
		\item If $C(f) = -2$, then $g(0) = -2$, $g^{(k)}(0)=2$ for all $k \ge 2$, and therefore $F\(g^{(2k-2)}(0)\)=a(2-b/2) $ for all $k \ge 2$.
	\end{itemize}
In these cases one can choose $a$ and $b$ for which 	$-ab/2$ or $a(2-b/2)$ is a square for all $k\ge 2$.		
\end{remark}

\subsection{Proof of Theorem \ref{thm:extend_for_semigroup_final}}
The proof follows similar ideas as the proof of Theorem~\ref{thm:pow}. For each $W\in \langle \mathcal{F}\rangle$, we define $\ell(W)=r$ if there exist $f_1,f_2,\cdots,f_r\in \mathcal{F}$ such that $W=f_1\circ f_2\circ\cdots\circ f_r$. We recall the notation of $E_\K(h)$ defined by~\eqref{eq:C_b}. 

If $k< 6+M$, where $M$ is defined  below, there is nothing to prove. 
We therefore assume that $k\ge 6+M$ and show that any composition $W_k \in \langle\mathcal{F}\rangle$ of length $k$ can be decomposed as $W_k = W_1 \circ W_2$, where the prefix $(W_1,m)$ satisfies the LeVeque condition and has length $\ell(W_1) \le 6$. Examining the first $6$ terms of $W_k$, we encounter  two cases:

\textbf{Case 1:} $W_k = f^{(6)} \circ W_k'$ for some $f \in \mathcal{F}$. 
	In this case, if $m=2$, since $|\Root_m(f)| \ge 2$, Lemma \ref{lem:second_iterate_parity} ensures that $f^{(2)}$ is not of the form $X^sP(X)^2$ where $s\in\set{0,1}$. Consequently, by Lemma \ref{lem:Young}, there exists an integer $j \le 6$ such that the iterate $(f^{(j)},m)$ satisfies the LeVeque condition. Thus, we obtain the claimed composition $W_k = W_1 \circ W_2$, where $W_1 = f^{(j)}$ and $W_2$ is the remainder of the word.
	
\textbf{Case 2:} $W_k=f^{(i)}\circ g\circ W_2$ for some $1\le i\le 5$, $g\in \cF$, $g\ne f$, and $W_2\in \langle\mathcal{F}\rangle$. Then we set $W_1=f^{(i)}\circ g$.
	
	By our hypothesis, $|\Root_m(f)| \ge 2$. Applying Corollary \ref{cor:preservation_iteration} inductively over the iterations of $f$, we have $|\Root_m(f^{( j)})| \ge 2$ for all $j\ge 1$.
	
	Since by the hypothesis $\Root_m(f^{(i)}) \cap \mathcal{C}_g = \emptyset$, applying Lemma \ref{lem:suff_LeVeque} to the composition $W_1 = f^{(i)} \circ g$, we obtain 
	\begin{align*}
		|\Root_m(W_1)| \ge d_g \cdot |\Root_m(f^{(i)})| \ge  4. 
	\end{align*}
	Therefore $(W_1,m)$.

In both cases, we have factored $W_k = W_1 \circ W_2$, where the prefix $(W_1,m)$  satisfies the LeVeque condition and its length satisfies $\ell(W_1) \le 6$. 

Returning to our initial equation $W_k(\alpha) = b\beta^m$, we rewrite it as $W_1(x) = b\beta^m$, where $x = W_2(\alpha)$. Since $\alpha$ and the polynomials in $\mathcal{F}$ are defined over $\mathcal{O}_S$, the  element $x$ is  an $S$-integer.

Let $\mathcal{W}_{\le 6}^{\mathrm{LeV}}$ be the set of all compositions in $\langle \mathcal{F} \rangle$ of length at most $6$ that satisfy the LeVeque condition with respect to $m$. Since the family $\mathcal{F}$ is finite, the set $\mathcal{W}_{\le 6}^{\mathrm{LeV}}$ is also finite. By our construction in Case 1 and Case 2, the prefix $W_1$ is  an element of $\mathcal{W}_{\le 6}^{\mathrm{LeV}}$.

For each  $W \in \mathcal{W}_{\le 6}^{\mathrm{LeV}}$ of degree $d_W $, Lemma~\ref{lem:T_LeVeque} provides an  upper bound  for the height of any solution $x \in \mathcal{O}_S$ to the equation $W(x) = b\beta^m$ of the form
$$
\h(x)\le C_W\cdot N_S(b)^{e(m,r_W)},
$$
 where $e(m,r_W)$ and the constant $C_W$ are defined by~\eqref{eq:bound222} and~\eqref{eq:bound32}, and thus depend only on $\K$, $\cF$, $m$ and $S$.

Let $C=\max_{W \in \mathcal{W}_{\le 6}^{\mathrm{Lev}}} C_W$ and $e=\max_{W \in \mathcal{W}_{\le 6}^{\mathrm{Lev}}} e(m,r_W)$. Then, since 
 $W_1 \in \mathcal{W}_{\le 6}^{\mathrm{Lev}}$, we  obtain the upper bound
\begin{align}\label{eq:upper_bound_semi_explicit}
	\h(W_2(\alpha)) = \h(x) \le CN_S(b)^e.
\end{align}

We now establish a lower bound for $\h(W_2(\alpha))$. By Lemma \ref{lem:height_transformation}, there exists a uniform constant $L=L_{\mathcal{F}} > 0$ such that for any $f \in \mathcal{F}$ and any $z \in \ov\K$,
\begin{equation}
\label{eq:height f(z)}
	h(f(z)) - L \ge d_{\min} \cdot (\h(z) - L).
\end{equation}

By induction, for any  $W \in \langle \mathcal{F} \rangle$ of length $\ell(W)$ and any $z \in \overline{\mathbb{K}}$, we obtain
\begin{align}\label{eq:lower_bound_refined}
	\h(W(z)) - L \ge d_{\min}^{\ell(W)} \cdot (\h(z) - L).
\end{align}

We define set 
\begin{align*}
	S_L = \{ z \in \mathbb{K} : 
	\h(z) \le L+1 \}.
\end{align*}
By Lemma~\ref{lem:boundofset of boundheight} we have $|S_L| \le M$, where $M=\lfloor E_\K(L+1)\rfloor$ is defined by~\eqref{eq:C_b}.

Suppose that there exists a composition $F = g_{M} \circ \dots \circ g_1 \in \langle \mathcal{F} \rangle$ of exact length $\ell(F) = M$ such that $\h(F(\alpha)) \le L+1$. Let $(y_r)_{r=0}^{M}$ be the finite sequence of successive orbit evaluations defined by $y_0 = \alpha$ and $y_r = (g_r \circ \dots \circ g_1)(\alpha)$ for $1 \le r \le M$. 

We claim that $y_r \in S_L$ for all $r \in \{0, 1, \dots, M\}$. Indeed, if there were an index $r$ such that $\h(y_r) > L+1$, applying \eqref{eq:lower_bound_refined} we obtain
\begin{align*}
	\h(F(\alpha)) - L \ge d_{\min}^{M - r} \cdot (\h(y_r) - L) > d_{\min}^{M - r},
\end{align*}
which contradicts the initial assumption that $h(F(\alpha)) \le L+1$. Thus, we have
\begin{align*}
	y_r \in S_L, \quad  r \in \{0, 1, \dots, M\}.
\end{align*}

Since the sequence $(y_r)_{r=0}^{M}$ contains exactly $M + 1$ elements and $|S_L| \le M$,  there must exist two distinct indices $a, b \in \mathbb{Z}$ such that $0 \le a < b \le M$ and $y_a = y_b$, which contradicts the hypothesis that $\alpha$ is wandering.

Therefore, $\h(F(\alpha)) > L+1$ holds for all $F \in \langle \mathcal{F} \rangle$ of length $M$. Since $\mathcal{F}$ is finite, the set of all words of length $M$ is also finite.

Since we have assumed that the total composition length $k \ge 6 + M$ (recall that $M=E_\K(L+1)$), we can partition $W_2$ as $W_2 = W_t \circ F$, where $\ell(F) = M$ and $\ell(W_t) = \ell(W_2) - M \ge k - 6 - M$. Setting $z = F(\alpha)$ and applying the inductive inequality \eqref{eq:lower_bound_refined} to $W_t(z)$, we deduce
\begin{align}\label{eq:lower_bound_final}
	\h(W_2(\alpha)) 
	&\ge d_{\min}^{\ell(W_t)} \cdot \left( \h(F(\alpha)) - L \right) + L \nonumber \\
	&\ge d_{\min}^{k - 6 - M}.
\end{align}

Combining~\eqref{eq:lower_bound_final} with~\eqref{eq:upper_bound_semi_explicit}, we obtain the inequality
\begin{align*}
	  d_{\min}^{k - 6 - M} \le CN_S(b)^e.
\end{align*}
Since $C$ is positive and completely determined by $\K, \mathcal{F},  m$, and $S$, we conclude the proof.

\subsection{Proof of Corollary~\ref{cor:multiplicative_dependence}}
When $\gamma$ is a root of unity, let $W(\alpha)\in \mathcal{O}_\mathcal{F}(\alpha)\cap \mathcal{M}_\gamma$. Then $W(\alpha)$ is also a root of unity. Since $\K$ has finitely many roots of unity and $\alpha$ is a wandering point, the result is trivial. We now assume that $\gamma $ is not a root of unity. 

	Assume that the intersection is infinite. Since the family $\mathcal{F}$ is finite, an infinite intersection implies the existence of an infinite sequence of compositions $W_k \in \langle \mathcal{F} \rangle$ with  increasing lengths $k$ such that $W_k(\alpha) \in \mathcal{M}(\gamma)$.
	
	For each such $x=W_k(\alpha) \in \mathcal{M}(\gamma)$ for some $k\ge 1$,  there exist $a,b\in \Z$ not both zero such that $x^a=\gamma^b$. Since $\gamma$ is not a root of unity, we may assume that $a>0$, and extracting the $\gcd(a,b)$-th root of the multiplicative relation, we may assume that  $\gcd(a, b) = 1$ and $x^a = \zeta\gamma^b$ for a  root of unity $\zeta\in\K$.
	
 By Bézout's identity, there exist integers $u$ and $v$ such that $ua + vb = 1$. Defining $z = \gamma^u x^v\in \mathbb{K}^\times$, we have
 \begin{equation}
 \label{eq:gamma pow}
 \gamma =\gamma^{ua+vb}= \gamma^{ua} \zeta^{-v} x^{va}=\zeta^{-v} z^a,
 	\end{equation}
	and
 \begin{equation}
 \label{eq:x pow}
x =x^{ua+vb}= \zeta^u\gamma^{ub} x^{vb} =\zeta^u z^b.
	\end{equation}
	
	We now take the logarithmic Weil height of~\eqref{eq:gamma pow} to uniformly bound $a$. The relation $\gamma =\zeta^{-v} z^a$ implies $h(\gamma) = a \cdot h(z)$. Since $\gamma$ is not a root of unity, $z$ is not a root of unity either, and by Kronecker's theorem $h(z) > 0$. Using Northcott's theorem one can show there exists a positive absolute constant $c(\mathbb{K}) > 0$, depending only on $\mathbb{K}$, such that the height of any non-torsion element in $\mathbb{K}^\times$ is bounded below by $c(\mathbb{K})$ (an explicit lower bound is given by Dobrowolski~\cite{Do}). Thus, $h(z) \ge c(\mathbb{K})$, yielding $a \le \frac{h(\gamma)}{c(\mathbb{K})}.$
	Consequently, setting $A_0 = \max \{ 2, \lfloor h(\gamma)/c(\mathbb{K}) \rfloor \}$, we have $a \le A_0$.
	
	Let $D = \max_{f \in \mathcal{F}} \deg(f)$. We  fix an integer $m> \max(A_0, D)$ such that $\gcd(\prod_{1\le a\le A_0}a,m)=1$. 
	Again, by Bézout's identity, there exist integers $u', v'$ such that $u'a + v'm = 1$. Using~\eqref{eq:gamma pow}, this gives
	\begin{align*}
		z = z^{u'a + v'm} = (z^a)^{u'} (z^{v'})^m =\zeta^{vu'} \gamma^{u'} (z^{v'})^m,
	\end{align*}
and whence, from~\eqref{eq:x pow}, we obtain
	\begin{align*}
x = \zeta^u\left(\zeta^{vu'}  \gamma^{u'} (z^{v'})^m \right)^b =\zeta^{vu'b+u} \gamma^{u'b} (z^{v'b})^m. 
	\end{align*}
	
	Writing $u'b = mq + r$ for some integers $q, r$ with $0 \le r < m$, we have
	\begin{align*}
		 x =\zeta^{vu'b+u} \gamma^{mq+r} (z^{v'b})^m =\zeta^{vu'b+u} \gamma^r \left( \gamma^q z^{v'b} \right)^m.
	\end{align*}
	
	Let $\beta = \gamma^q z^{v'b}\in \mathbb{K}^\times$. Thus $W_k(\alpha) = C_r \beta^m$, where the constant $C_r =\zeta^{vu'b+u}\gamma^r$ belongs to a finite subset of $\mathbb{K}^\times$.
	
	Since $m > D \ge \deg(f)$ for all $f \in \mathcal{F}$, the multiplicity of any root of any polynomial $f \in \mathcal{F}$ is bounded by $D < m$. Thus, $m$ does not divide the multiplicity of any root of $f$. This implies that $\Root_m(f) = \cR(f)$ for all $f \in \mathcal{F}$. 
	
	By our hypothesis, $|\cR(f)| \ge 2$ for any $f\in \mathcal{F}$, which guarantees $|\Root_m(f)| \ge 2$. This also implies that \begin{align*}
		\Root_m(f^{(k)}) \cap \mathcal{C}_g\subseteq\mathcal{R}(f^{(k)}) \cap \mathcal{C}_g = \emptyset
	\end{align*}  for any $1\le k\le 5$.
	
We extend the set $S$ to a set of places $S'$ such that $\gamma\in\cO_{S'}$.	Therefore,  the family $\mathcal{F}$ and the chosen $m$ satisfy all hypotheses of Theorem \ref{thm:extend_for_semigroup_final}. Applying Theorem \ref{thm:extend_for_semigroup_final} to the equation $W_k(\alpha) = C_r \beta^m$, $C_r,\beta\in\cO_{S'}$, (and taking the maximum  bound over the finitely many choices of $C_r $) yields an upper bound on the word length
	\begin{align*}
		k \ll_{\K,\mathcal{F}, m, \alpha, \gamma} 1. 
	\end{align*}
	This contradicts our initial assumption of an infinite sequence of words with increasing lengths $k$. Thus, the intersection $\mathcal{O}_{\mathcal{F}}(\alpha) \cap \mathcal{M}(\gamma)$ must be finite.

\end{document}